\documentclass[12pt,english,a4paper,twoside]{article}

\usepackage{graphicx}
\usepackage[hidelinks,linktocpage]{hyperref}  
\hypersetup{hypertexnames=false}
\usepackage{lmodern}
\usepackage[left=3cm,right=3cm,top=3cm,bottom=3cm]{geometry} 
\usepackage{url}                   
\usepackage{booktabs}

\usepackage{amsthm}
\usepackage{amsmath}
\usepackage{amsfonts}
\usepackage{dsfont}                
\usepackage{stix}                  
\usepackage{mathtools}
\usepackage{xparse, etoolbox}
\usepackage{enumerate}
\usepackage{mathrsfs}
\usepackage{array}
\usepackage{enumitem}
\usepackage{fancyhdr}              
\usepackage{subdepth}              
\usepackage{titlesec}              
\usepackage{xypic}    
\usepackage{xfrac}   
\usepackage[numbers,sort]{natbib}
\usepackage[ruled]{algorithm2e}
\SetKwInput{KwInput}{input}        
\SetKwInput{KwOutput}{output}      
\SetKwBlock{Repeat}{repeat}{} 
\usepackage[overload]{empheq}     

\titleformat*{\paragraph}{\itshape}

\makeatletter
\let\amstexbig\big
\def\newbig#1{
  \ifx#1|
    \expandafter\@firstoftwo
  \else
    \expandafter\@secondoftwo
  \fi
  {\big@bar}
  {\amstexbig{#1}}
}
\AtBeginDocument{\let\big\newbig}
\def\big@bar{\bBigg@{1.1}|}
\makeatother

\def\SZ{\mathscr{Z}}

\def\CH{\mathcal{H}}
\def\CI{\mathcal{I}}

\def\CO{\mathcal{O}}

\def\BI{\mathbb{I}}

\def\BN{\mathbb{N}}

\def\BR{\mathbb{R}}

\def\BNz{\BN_{0}}

\def\BfA{\mathbf{A}}
\def\BfI{\mathbf{I}}
\def\BfK{\mathbf{K}}
\def\BfM{\mathbf{M}}
\def\BfQ{\mathbf{Q}}

\def\BfW{\mathbf{W}}
\def\BfX{\mathbf{X}}

\def\alpb{\boldsymbol{\alpha}}

\def\vepsb{\boldsymbol{\varepsilon}}

\def\bde{\boldsymbol{e}}
\def\bdc{\boldsymbol{c}}

\def\bdf{\boldsymbol{f}}
\def\bdh{\boldsymbol{h}}

\def\bdu{\boldsymbol{u}}
\def\bdv{\boldsymbol{v}}
\def\bdx{\boldsymbol{x}}
\def\bdy{\boldsymbol{y}}
\def\bdz{\boldsymbol{z}}

\def\BI{\mathrm{BI}}
\def\up{\mathrm{up}}
\def\correc{\mathcal{A}}
\def\correcT{\mathfrak{a}}
\def\coneC{\mathscr{C}}
\def\up{\mathrm{up}}
\def\one{\mathds{1}}

\DeclareMathOperator{\vspan}{span}

\DeclareMathOperator{\sign}{sign}
\DeclareMathOperator{\coni}{coni}

\DeclarePairedDelimiterX\mip[1]\langle\rangle{\mipargs{#1}}
\NewDocumentCommand{\mipargs}{>{\SplitArgument{1}{,}}m }
 {\mipargsaux#1}
\NewDocumentCommand{\mipargsaux}{ m m }
{\ifblank{#1}{\cdot}{#1}\nonscript\,\delimsize\vert\nonscript\,\mathopen{}\ifblank{#2}{\cdot}{#2}}

\theoremstyle{plain}
\newtheorem{theorem}{Theorem}[section]
\newtheorem{corollary}{Corollary}[section]

\newtheorem{lemma}{Lemma}[section]

\theoremstyle{definition}
\newtheorem{remark}{Remark}[section]
\newtheorem{example}{Example}[section]

\newcommand{\fin} {\hfill\hbox{$\triangleleft$}}

\title{Rescaling and unconstrained minimisation of convex quadratic maps}

\author{Alexandra \textsc{Zverovich}\footnotemark[1]\qquad  
Matthew \textsc{Hutchings}\footnotemark[1]\\
Bertrand \textsc{Gauthier}\footnotemark[1]}

\date{} 

\newcommand\shorttitle{Rescaling and quadratic minimisation}
\newcommand\authors{A. \textsc{Zverovich}, M. \textsc{Hutchings} and B. \textsc{Gauthier} }

\begin{document}

\maketitle

\renewcommand{\thefootnote}{\fnsymbol{footnote}}
\footnotetext[1]{Cardiff University, School of Mathematics\\
\hspace*{1.4em}\hspace*{1ex}Abacws, Senghennydd Road, Cardiff, CF24 4AG, United Kingdom\\
\hspace*{1.4em}\hspace*{1ex}{ZverovichA@cardiff.ac.uk}, 
{HutchingsM1@cardiff.ac.uk}, 
{GauthierB@cardiff.ac.uk}}
\renewcommand{\thefootnote}{\arabic{footnote}}

\begin{abstract}
We investigate the properties of a class of piecewise-fractional maps arising from the introduction of
an invariance under rescaling into convex quadratic maps. 
The subsequent maps are quasiconvex,   
and pseudoconvex on specific convex cones; 
they can be optimised via exact line search along admissible directions, 
and the iterates then inherit a bidimensional optimality property. 
We study the minimisation of such relaxed maps 
via coordinate descents with gradient-based rules,  
placing a special emphasis on coordinate directions 
verifying a maximum-alignment property
in the reproducing kernel Hilbert spaces related to the underlying positive-semidefinite matrices.
In this setting, we illustrate that accounting for the optimal rescaling of the iterates can in certain situations
substantially accelerate the unconstrained minimisation of convex quadratic maps. 
\end{abstract}


\vspace{0.25\baselineskip}

\tableofcontents

\section{Introduction}
\label{sec:Introduction}
The unconstrained minimisation of convex quadratic maps is 
one of the most fundamental tasks in optimisation and scientific computing;
it is equivalent to the solving of systems of linear equations 
defined by symmetric positive-semidefinite (SPSD) matrices, 
and as such plays a central role in numerical linear algebra. 
Although such problems can be solved via direct approaches 
(using Cholesky decomposition, for instance), 
the time and space complexity of direct solvers often prevents 
their application to problems of very large scale.  
In such situations, a common alternative consists of relying on iterative solvers;  
classical representatives of this type of approach
are the Gauss-Seidel or conjugate-gradient (CG) methods, for instance 
(see e.g. \cite{nocedal1999numerical, saad2003iterative, chong2004introduction, allaire2008numerical, gower2015randomized, Hackbusch2016}). 

Although CG-based solvers enjoy excellent properties, 
their iterations involve matrix-vector products which, 
for a matrix of order $N$, have an $\CO(N^{2})$ worst-case time complexity. 
In contrast, strategies based on coordinate descent (CD), 
where a single coordinate is updated at a time via exact line search, 
have an $\CO(N)$ worst-case computational cost per iteration.  
This difference in complexity makes CD-based approaches interesting candidates for very large problems, 
especially when high precision is not required 
or when the problem structure 
allows for rapid convergence through selective coordinate updates, 
see e.g. \cite{lee2013efficient, ma2015convergence, hefny2017rows, gordon2023conjugate}. 

In this work, we investigate the properties of a class of CD-type solvers 
for unconstrained quadratic minimisation
which leverages the introduction of an invariance under rescaling 
into the underlying quadratic maps.
More precisely, rather than directly minimising a convex quadratic map $D:\BR^{N}\to\BR$, 
following \cite{hutchings2023energy}, 
we consider the \emph{relaxed map} 
${R:\BR^{N}\to\BR}$ such that  
\[
R(\bdx)=D(s_{\bdx}\bdx)=\min_{s\geqslant0}D(s\bdx), 
\bdx\in\BR^{N},   
\]
with $s_{\bdx}$ the \emph{optimal non-negative rescaling} of $\bdx$. 
The map $R$ is \emph{invariant under rescaling}, 
that is, it verifies $R(s\bdx)=R(\bdx)$, $s>0$. 
This map is an instance of a \emph{piecewise-fractional map}
(see e.g. \cite{horst2013handbook}); 
it is quasiconvex on $\BR^{N}$, 
and pseudoconvex on a specific convex cone 
(see Theorem~\ref{thm:RmapPseudoCvx}).  
It can be minimised via exact line search along admissible directions, 
and the iterates then inherit a bidimensional optimality property 
(see Theorem~\ref{thm:RmapOptSS}). 
The pseudoconvex component of $R$ relates to the ratio
between the squared linear term of $D$ 
and its quadratic term, 
see Section~\ref{sec:frameAndNot}. 

We study the unconstrained minimisation of $R$ 
via exact CD with gra\-dient-based rules, 
and compare the properties of the considered strategies 
with similar strategies applied to the minimisation of $D$.  
We place a special emphasis on the notion of \emph{$\CH$ coordinates},  
which correspond to the coordinate directions 
whose potentials (see Section~\ref{sec:frameAndNot}) 
align the most with the gradients of $D$ or $R$ 
in the reproducing kernel Hilbert space (RKHS)
related to the underlying SPSD matrix. 
In this framework, 
we illustrate that accounting for the optimal rescaling of the iterates
by minimising $R$ instead of $D$
can in certain situations substantially accelerate the unconstrained 
minimisation of convex quadratic maps.    
We in particular derive upper bounds on the convergence rates of the considered strategies
which support the existence, under specific conditions, of a rescaling-induced 
asymptotic acceleration phenomenon (see Theorem~\ref{thm:AsympCorr});  
the effectivity of this asymptotic acceleration is demonstrated on a series of examples. 
  
The manuscript is organised as follows. 
Section~\ref{sec:frameAndNot} introduces the general framework
and the main notations of the study. 
The properties of the relaxed map $R$ are investigated in Section~\ref{sec:PropOfR}, 
and Section~\ref{sec:CoordDescGradRule} discusses the minimisation of $D$ and $R$ 
via exact CD with gradient-based rules. 
Section~\ref{sec:Experiments} is dedicated to numerical experiments, 
and Section~\ref{sec:ConcluDiscuss} comprises a concluding discussion. 

\section{Framework and notations}
\label{sec:frameAndNot}
Throughout this note, we use the classical \emph{matrix notation} 
and identify a vector $\bdx\in\BR^{N}$, $N\in\BN$,  
as the $N\times 1$ column matrix defined by the coefficients 
of $\bdx$ in the canonical basis $\{\bde_{i}\}_{i\in[N]}$ of $\BR^{N}$; 
$[N]$ stands for the set of all integers between $1$ and $N$. 
Unless otherwise stated, we consider the standard topology of $\BR^{N}$.  
The transpose of a matrix $\BfM$ is denoted by $\BfM^{T}$, 
$\BfM^{\dag}$ is the pseudoinverse (Moore-Penrose inverse) of $\BfM$, 
and $\vspan\{\BfM\}$ is the linear space spanned by the columns of $\BfM$. 
For an SPSD matrix ${\BfA\in\BR^{N\times N}}$, 
we define the bilinear form $\mip[]{\bdx,\bdy}_{\BfA}=\bdx^{T}\BfA\bdy$, $\bdx$ and $\bdy\in\BR^{N}$,  
and we denote by $\|\bdx\|_{\BfA}=\mip[]{\bdx,\bdx}_{\BfA}^{1/2}$ the associated semi-norm.  
For $S\subseteq\BR^{N}$, $\coni(S)$ stands for the conical hull of $S$. 

\paragraph{Quadratic map.}
We consider an SPSD matrix $\BfQ\in\BR^{N\times N}$
and a vector $\bdc\in\vspan\{\BfQ\}$, with $N\in\BN$. 
Let ${\alpb\in\BR^{N}}$ be such that $\BfQ\alpb=\bdc$ 
(in practical situations, $\alpb$ is unknown); 
observe that 
$\alpb=\BfQ^{\dag}\bdc+\vepsb$, 
with $\vepsb\in\SZ=\{\bdx\in \BR^{N}  |\; \BfQ\bdx=0\}$.  

We define the convex quadratic map $D:\BR^{N}\to\BR$, 
given by
\begin{equation*}
D(\bdx)
=\bdx^{T}\BfQ\bdx - 2 \bdc^{T}\bdx  
+\bdc^{T}\alpb, 
\bdx\in\BR^{N}.     
\end{equation*}
Classically, the constant term 
$\bdc^{T}\alpb=\alpb^{T}\BfQ\alpb=\bdc^{T}\BfQ^{\dag}\bdc$ 
is introduced for analytical purposes only;  
it is such that  
\[
D(\bdx)
=\|\bdx-\alpb\|_{\BfQ}^{2}
=\|\BfQ\bdx-\bdc\|_{\BfQ^{\dag}}^{2}.    
\]  
For simplicity and without loss of generality, we assume that $\bdc\neq0$ (and so $\BfQ\neq0$), 
the case $\bdc=0$ being of no practical interest in the framework of this study. 

\paragraph{Relaxed map.}
Following \cite{hutchings2023energy}, 
we consider the map $R:\BR^{N}\to\BR$ defined as
\begin{align*}
R(\bdx)
=\min_{s\geqslant 0}D(s\bdx)
=\begin{dcases}
\text{$\bdc^{T}\alpb-(\bdc^{T}\bdx)^{2}/(\bdx^{T}\BfQ\bdx)$ if $\bdx\in\coneC$, }\\
\text{$\bdc^{T}\alpb$ otherwise,  }
\end{dcases}
\end{align*}
with $\coneC=\{\bdx \in \BR^{N}  | \; \bdc^{T}\bdx > 0\}$.  
From the Cauchy-Schwartz (CS) inequality, observe that 
$|\bdc^{T}\bdx|^{2}=|\alpb^{T}\BfQ\bdx|^{2}\leqslant(\alpb^{T}\BfQ\alpb)(\bdx^{T}\BfQ\bdx)$,  
so that $\coneC\cap\SZ=\emptyset$. 
We have 
\begin{align*}
R(\bdx)=D(s_{\bdx}\bdx), 
\quad\text{with}\quad
s_{\bdx}
=\begin{dcases}
\text{$(\bdc^{T}\bdx)/(\bdx^{T}\BfQ\bdx)$ if $\bdx\in\coneC$, }\\
\text{$0$ otherwise.  }  
\end{dcases}
\end{align*}
The relaxed map $R$ is invariant under rescaling, 
that is, ${R(s\bdx)=R(\bdx)}$, $\bdx\in\BR^{N}$ and $s>0$, 
and $s_{\bdx}\bdx$ corresponds to the 
optimal non-negative rescaling of $\bdx$.  
We have $0\leqslant R(\bdx)\leqslant \bdc^{T}\alpb$, 
$R(\bdx)\leqslant D(\bdx)$, 
and $R$ is maximum outside of $\coneC$. 
Observe that 
$R$ is an instance of a piecewise-fractional map 
(see e.g. \cite{horst2013handbook});   
for $N=1$, $R$ is piecewise-constant.  

To study the properties of $D$ and $R$, 
and more specifically their minimisation using gradient-based strategies,  
it is convenient to introduce the RKHS related to $\BfQ$. 

\paragraph{RKHS related to an SPSD matrix.}\label{rem:Q-RKHS}
The entries of the matrix $\BfQ$ characterise the kernel of an RKHS
of \mbox{$\BR$-valued} functions on $[N]$. 
This RKHS can be identified with the subspace ${\CH=\vspan\{\BfQ\}\subseteq\BR^{N}}$ 
endowed with the inner product  
$(\bdh,\bdf)\mapsto\mip[]{\bdh,\bdf}_{\BfQ^{\dag}}$, 
$\bdh$ and ${\bdf\in\CH}$, 
so that the RKHS norm corresponds 
to the restriction to $\CH$ of the seminorm $\|.\|_{\BfQ^{\dag}}$; 
see for instance \cite[Chapter~2]{paulsen2016introduction}. 
In this framework, for $\bdh\in\CH$, 
setting $\bdh=\BfQ\bdx$, $\bdx\in\BR^{N}$, 
the \emph{reproducing property} reads
\[
\mip[]{\BfQ\bde_{i},\bdh}_{\BfQ^{\dag}}
=\bde_{i}^{T}\BfQ\BfQ^{\dag}\BfQ\bdx
=\bde_{i}^{T}\bdh, i\in[N].   
\]
For $\bdu\in\BR^{N}$, we more generally have
$\mip[]{\BfQ\bdu,\bdh}_{\BfQ^{\dag}}
=\bdu^{T}\bdh$,     
and by analogy with the literature on the kernel embedding of measures,
we refer to $\BfQ\bdu$ as the \emph{potential} of $\bdu$ in $\CH$; 
we denote by $P_{\BfQ\bdu}$ the orthogonal projection from 
$\CH$ onto $\vspan\{\BfQ\bdu\}$.   
For $\bdx\in\coneC\cup\SZ$, we obtain 
$s_{\bdx}\BfQ\bdx=P_{\BfQ\bdx}\bdc$, and 
$R(\bdx)=\|\bdc-P_{\BfQ\bdx}\bdc\|_{\BfQ^{\dag}}^{2}$.  

\section{Properties of the relaxed map}
\label{sec:PropOfR}
The directional derivative $\Lambda(\bdx;\bdv)$ of $R$ 
at $\bdx\in\BR^{N}$ and along $\bdv\in\BR^{N}$ is 
\[
\Lambda(\bdx;\bdv)=
\lim_{t \to 0^{+}} \frac{1}{t} \big[R(\bdx + t\bdv) - R(\bdx)\big]
=\begin{dcases}
\text{$-\infty$ if $\bdx\in\SZ$ and $\bdv\in\coneC$, }\\
\text{$2s_{\bdx}\bdv^{T}(s_{\bdx}\BfQ\bdx - \bdc)$ otherwise. }  
\end{dcases}
\]
Observe that the invariance of $R$ under rescaling translates into the equality $\Lambda(\bdx;\bdx)=0$, $\bdx\in\BR^{N}$. 
The gradient (with respect to the Euclidean structure of $\BR^{N}$) 
of $R$ at ${\bdx\not\in\SZ}$ is ${\nabla R(\bdx)=2s_{\bdx}(s_{\bdx}\BfQ\bdx - \bdc)}$,  
and we have $\nabla R(\bdx)\in\CH$. 

\begin{remark}\label{rem:HessianOfR}
The Hessian of $R$ at $\bdx\in\coneC$ is 
\[
\nabla^{2}R(\bdx)
=2\Big[s_{\bdx}^{2}\BfQ
-\frac{1}{\bdx^{T}\BfQ\bdx}(2s_{\bdx}\BfQ\bdx-\bdc)(2s_{\bdx}\BfQ\bdx-\bdc)^{T}\Big];  
\]
this matrix admits at most one negative eigenvalue. 
\fin
\end{remark}

\begin{theorem}\label{thm:RmapPseudoCvx}
The map $R$ is quasiconvex on $\BR^{N}$, 
and pseudoconvex on the convex \mbox{cone $\coneC$}. 
\end{theorem}

 A proof of Theorem~\ref{thm:RmapPseudoCvx} can be found in \cite{hutchings2023energy}; 
for completeness, a proof is reproduced in Appendix~\ref{sec:ProofThmPseudoCvx}. 
  
The forthcoming Theorem~\ref{thm:RmapOptSS} characterises
the directions along which $R$ can be minimised via exact line search. 
Notably, due to the invariance of $R$ under rescaling, 
the iterate of an exact line search from $\bdx$ along $\bdv$ 
minimises $R$ over $\vspan\{\bdx,\bdv\}$. 
To simplify the notations, for $\bdx$ and $\bdv\in\BR^{N}$, we introduce 
\[
\Upsilon(\bdx;\bdv)
=(\bdc^{T}\bdv)(\bdx^{T}\BfQ\bdx)-(\bdc^{T}\bdx)(\bdv^{T}\BfQ\bdx).  
\]  
Observe that for $\bdx\in\coneC$, we have $\Upsilon(\bdx;\bdv)=-\Lambda(\bdx;\bdv)(\bdx^{T}\BfQ\bdx)/(2s_{\bdx})$. 
A schematic representation of the situations described in Theorem~\ref{thm:RmapOptSS} 
is provided in Figure~\ref{fig:Illustr_Thm32}. 

\begin{theorem}\label{thm:RmapOptSS}
Consider $\bdx\in\coneC$ and $\bdv\in\BR^{N}$;  
set $\bdz_{t}=\bdx+t\bdv$, $t\in\BR$. 
If $\BfQ\bdx$ and $\BfQ\bdv$ are non-collinear, 
the following assertions hold.  
\begin{enumerate}[label=(\roman*)]
\item\label{item:OSSDesc} If $\Upsilon(\bdv;\bdx)>0$, 
then the function $t\mapsto R(\bdz_{t})$, $t\in\BR$, is minimum at 
\[
\tau=\Upsilon(\bdx;\bdv)/\Upsilon(\bdv;\bdx);
\]     
we in this case have $\bdz_{\tau}\in\coneC$ and $R(\bdz_{\tau})=\min\limits_{\bdz\in\vspan\{\bdx,\bdv\}}R(\bdz)$.  

\item\label{item:NonStopDesc} If $\Upsilon(\bdv;\bdx)\leqslant0$, 
then the function $t\mapsto R(\bdz_{t})$, $t\in\BR$, is monotonic, 
and  
\[
\inf_{t\in\BR}R(\bdz_{t})
=\min\{R(-\bdv),R(\bdv)\};
\]
in particular, if $\Upsilon(\bdv;\bdx)=0$, 
then $\min\limits_{\bdz\in\vspan\{\bdx,\bdv\}}R(\bdz)=\min\{R(-\bdv),R(\bdv)\}$. 
\end{enumerate}
If $\BfQ\bdx$ and $\BfQ\bdv$ are collinear, then the map $R$ is piecewise-constant 
over $\vspan\{\bdx,\bdv\}$, taking the values $R(0)$ or $R(\bdx)<R(0)$ 
(and $t\mapsto R(\bdz_{t})$ is thus minimum at $t=0$). 
\end{theorem}

To prove Theorem~\ref{thm:RmapOptSS}, 
we rely on Lemmas~\ref{lem:MinRSpanZeroDir} and \ref{lem:MinRSpan}. 

\begin{lemma}\label{lem:MinRSpanZeroDir}
Consider $\bdx\in\coneC$ and $\bdv\in\BR^{N}$;  
if $\Lambda(\bdx;\bdv)=0$, then 
$R(\bdx)=\min\limits_{\bdz\in\vspan\{\bdx,\bdv\}}R(\bdz)$. 
\end{lemma}

\begin{proof}
By invariance under rescaling, we also have
$\Lambda(\bdx;\bdx)=0$, and the result holds by pseudoconvexity.   
\end{proof}

\begin{lemma}\label{lem:MinRSpan}
Consider $\bdx$ and $\bdv\in\BR^{N}$;  
assume that $\BfQ\bdx$ and $\BfQ\bdv$ are non-collinear, 
and that either $\bdc^{T}\bdx\neq 0$ or $\bdc^{T}\bdv\neq0$. 
We have
\[
\arg\min\limits_{\bdz\in\vspan\{\bdx,\bdv\}}R(\bdz)
=\{s\bdz_{\bdx,\bdv}|s>0\},   
\]
where $\bdz_{\bdx,\bdv}\in\coneC$ is given by
$\bdz_{\bdx,\bdv}=\Upsilon(\bdv;\bdx)\bdx+\Upsilon(\bdx;\bdv)\bdv$. 
\end{lemma}

\begin{proof}
Cancelling the partial derivatives of the quadratic map 
$(\beta,\gamma)\mapsto D(\beta\bdx+\gamma\bdv)$, 
$\beta$ and $\gamma\in\BR$, 
leads to the linear system 
\begin{alignat*}{1}[left = \empheqlbrace]
\beta(\bdx^{T}\BfQ\bdx) + \gamma (\bdv^{T}\BfQ\bdx) & = \bdc^{T}\bdx, \\
\beta(\bdu^{T}\BfQ\bdx) + \gamma (\bdv^{T}\BfQ\bdu) & = \bdc^{T}\bdv.  
\end{alignat*}
By CS, if $\BfQ\bdx$ and $\BfQ\bdv$ are non-collinear, then 
$\omega=(\bdx^{T}\BfQ\bdx)(\bdv^{T}\BfQ\bdv) - (\bdv^{T}\BfQ\bdx)^{2}>0$. 
In this case, a unique solution $(\bar{\beta}, \bar{\gamma})$ exists, with
\[
(\bar{\beta}, \bar{\gamma})=\big(\Upsilon(\bdv;\bdx),  \Upsilon(\bdx;\bdv)\big)/\omega, 
\] 
that is, $\bar{\beta}\bdx+\bar{\gamma}\bdv=\bdz_{\bdx,\bdv}/\omega$. 
As $(\bdc^{T}\bdx,\bdc^{T}\bdv)\neq (0,0)$, we have $\bdz_{\bdx,\bdv}\neq 0$.
Observing that $R(\bdz_{\bdx,\bdv})=D(\bdz_{\bdx,\bdv}/\omega)<D(0)=R(0)$, 
we get $\bdz_{\bdx,\bdv}\in\coneC$, 
and by invariance under rescaling, the minimum of $R$ over 
$\vspan\{\bdx,\bdv\}$ is reached over the unpointed ray spanned by $\bdz_{\bdx,\bdv}$. 
\end{proof}

\begin{proof}[Proof of Theorem~\ref{thm:RmapOptSS}]
If $\bdx\in\coneC$ and $\bdv\in\BR^{N}$ are such that 
$\BfQ\bdx$ and $\BfQ\bdv$ are non-collinear, 
then, by CS, we have $(\bdx^{T}\BfQ\bdx)(\bdv^{T}\BfQ\bdv) - (\bdv^{T}\BfQ\bdx)^{2}>0$, 
and so ${\bdz_{t}^{T}\BfQ\bdz_{t}^{}>0}$, $t\in\BR$.  
Also, as $\bdc^{T}\bdx>0$, 
Lemma~\ref{lem:MinRSpan} holds (with $\bdz_{\bdx,\bdv}\in\coneC$). 

Considering the intersection between the affine lines
$\{s\bdz_{\bdx,\bdv}|s\in\BR\}$ and $\{\bdz_{t}|t\in\BR\}$, 
we obtain the linear system of equations
\begin{alignat*}{1}[left = \empheqlbrace]
1 - s\Upsilon(\bdv;\bdx) & = 0, \\
t - s\Upsilon(\bdx;\bdv) & = 0.  
\end{alignat*}
For $\Upsilon(\bdv;\bdx)\neq0$, a unique solution $(\bar{s},{\bar{t}})$ exists, 
with
\[
\bar{s}=1/\Upsilon(\bdv;\bdx) 
\quad\text{and}\quad
\bar{t}=\Upsilon(\bdx;\bdv)/\Upsilon(\bdv;\bdx), 
\]
and for $\Upsilon(\bdv;\bdx)=0$, there are no solutions. 

For $\Upsilon(\bdv;\bdx)>0$, we have $\bar{s}>0$; 
in this case, the unpointed ray $\{s\bdz_{\bdx,\bdv}|s>0\}$ 
intersects with the line $\{\bdz_{t}|t\in\BR\}$
(see Figure~\ref{fig:Illustr_Thm32}),  
and Lemma~\ref{lem:MinRSpan} gives  
$R(\bdz_{\tau})=\min\limits_{\bdz\in\vspan\{\bdx,\bdv\}}R(\bdz)$,  
with $\tau=\bar{t}$, proving assertion~\ref{item:OSSDesc}.  

If $\Upsilon(\bdv;\bdx)\leqslant0$,  
then the unpointed ray $\{s\bdz_{\bdx,\bdv}|s>0\}$ 
and the line $\{\bdz_{t}|t\in\BR\}$ do not intersect  
(see Figure~\ref{fig:Illustr_Thm32}),  
and from Lemmas~\ref{lem:MinRSpanZeroDir} and \ref{lem:MinRSpan}, 
we necessarily have $\Lambda(\bdx;\bdv)\neq0$
(we would otherwise have $\bdx=\bdz_{0}\in\{s\bdz_{\bdx,\bdv}|s>0\}$).  
Set $f(t)=R(\bdz_{t})$, $t\in\BR$, and introduce
\begin{equation}\label{eq:DefgFunc}
g(t)=
2\big(t(\bdc^{T}\bdv)+(\bdc^{T}\bdx)\big)
\big(t\Upsilon(\bdv;\bdx)-\Upsilon(\bdx;\bdv)\big) 
/(\bdz_{t}^{T}\BfQ\bdz_{t}^{})^{2};   
\end{equation}
we have $f'(t)=g(t)$ if $\bdz_{t}\in\coneC$, and $f'(t)=0$ otherwise;  
observe that if $\bdc^{T}\bdv=0$, then $\Upsilon(\bdv;\bdx)>0$;  
we thus have $\bdc^{T}\bdv\neq0$, and for $t_{1}=-(\bdc^{T}\bdx)/(\bdc^{T}\bdv)$, 
we obtain $\bdc^{T}\bdz_{t_{1}}=0$ and $g(t_{1})=0$.     
The function $f$ is therefore continuously differentiable on $\BR$, 
and $f$ is maximum at $t_{1}$;  
remark that the set $\{t\in\BR|\bdz_{t}\in\coneC\}$ is an interval of the form 
$(-\infty,t_{1})$ or $(t_{1},+\infty)$. 
We then observe that we necessarily have $g(t)\neq 0$ for all $t\in\BR$ such that $\bdz_{t}\in\coneC$.  
Indeed, if there were $t_{2}\in\BR$ 
such that $\bdz_{t_{2}}\in\coneC$ and $g(t_{2})=0$, 
then $t_{2}$ would be unique from \eqref{eq:DefgFunc}, 
and such that $\Lambda(\bdz_{t_{2}};\bdv)=0$;  
consequently, Lemmas~\ref{lem:MinRSpanZeroDir} and \ref{lem:MinRSpan} 
would imply $\bdz_{t_{2}}\in\{s\bdz_{\bdx,\bdv}|s>0\}$, 
leading to a contradiction. 
The function $f$ is therefore monotonic.  
By invariance under rescaling, 
we next observe that 
\[
\inf_{t\in\BR}R(\bdz_{t})
=\inf_{\bdz\in\coni\{\bdz_{t}|t\in\BR\}}R(\bdz). 
\]
The extreme rays of the closure in $\BR^{N}$ of 
the convex cone $\coni\{\bdz_{t}|t\in\BR\}$
are spanned by the vectors $-\bdv$ and $\bdv$, 
entailing
$\inf_{t\in\BR}R(\bdz_{t})=\min\{R(-\bdv),R(\bdv)\}$; 
more precisely, the infimum is $R(\bdv)$ if $\bdv\in\coneC$, and $R(-\bdv)$ otherwise.  
In particular, 
if $\Upsilon(\bdv;\bdx)=0$, 
then $\bdz_{\bdx,\bdv}=\Upsilon(\bdx;\bdv)\bdv$, 
concluding the proof of assertion~\ref{item:NonStopDesc}. 

If $\BfQ\bdv=\beta\BfQ\bdx$, $\beta\in\BR$, then 
$R(\gamma\bdx+\delta\bdv)=R\big((\gamma+\delta\beta)\bdx\big)$, 
$\gamma$ and $\delta\in\BR$, 
concluding the proof.  
\end{proof}

\begin{figure}[htbp]
\centering
\includegraphics[width=0.75\linewidth]{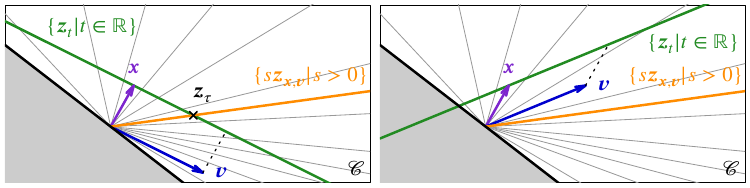}
\caption{Schematic representation of the situations discussed 
in Theorem~\ref{thm:RmapOptSS}.  
The left plot corresponds to the case $\Upsilon(\bdv;\bdx)>0$, 
and the right plot to $\Upsilon(\bdv;\bdx)\leqslant 0$. 
In each plot, the grey region indicates the set ${\{\bdx\in\BR^{N}|\bdc^{T}\bdx\leqslant 0\}}$, 
and the  grey lines are level sets of the map $R$ on $\vspan\{\bdx,\bdv\}$; 
the vector $\bdz_{\bdx,\bdv}$ is characterised in Lemma~\ref{lem:MinRSpan}. }
\label{fig:Illustr_Thm32}
\end{figure}

Corollary~\ref{cor:AboutRxLessThanRu} provides a sufficient condition 
for a line search from $\bdx$ along $\bdv$ to occur in the framework of 
Theorem~\ref{thm:RmapOptSS}-\ref{item:OSSDesc}. 

\begin{corollary}\label{cor:AboutRxLessThanRu}
If $\bdx\in\coneC$ and $\bdv\in\BR^{N}$  
are such that 
$R(\bdx)\leqslant\min\{R(-\bdv),R(\bdv)\}$ 
and $\Lambda(\bdx;\bdv)\neq0$, 
then $\Upsilon(\bdv;\bdx)>0$. 
\end{corollary}

\begin{proof} 
As $\Lambda(\bdx;\bdv)\neq0$, 
$\BfQ\bdx$ and $\BfQ\bdv$ are non-collinear, 
and the function $f:t\mapsto R(\bdz_{t})$ is non-constant 
and continuously differentiable on $\BR$; 
also, if $f$ admits a minimum, then its argument is unique 
(see the proof of Theorem~\ref{thm:RmapOptSS}).  
As $R(\bdx)\leqslant\min\{R(-\bdv),R(\bdv)\}$, 
observing that 
$\lim_{t\to\pm\infty}f(t)=R(\pm\bdv)$, by quasiconvexity,  
$f$ necessarily reaches its infimum, 
and the corresponding line search thus occurs 
in the framework of Theorem~\ref{thm:RmapOptSS}-\ref{item:OSSDesc}. 
\end{proof}

Lemma~\ref{lem:ImprovVal} provides an expression for the improvement 
yielded by an exact line search from $\bdx$ along $\bdv$. 
Following Theorem~\ref{thm:RmapOptSS}, we more generally introduce 
\[
\CI_{R}(\bdx;\bdv)=R(\bdx)-\min\limits_{\bdz\in\vspan\{\bdx,\bdv\}}R(\bdz), 
\text{$\bdx$ and $\bdv\in\BR^{N}$. }
\] 

\begin{lemma}\label{lem:ImprovVal}
Consider $\bdx\in\coneC$ and $\bdv\in\BR^{N}$.     
If $\BfQ\bdx$ and $\BfQ\bdv$ are non-collinear, 
we have 
\begin{align}\label{eq:OSSImprovVal}
\CI_{R}(\bdx;\bdv)
& =\big(\bdv^{T}(s_{\bdx}\BfQ\bdx-\bdc)\big)^{2}\big/
\big((\bdv^{T}\BfQ\bdv)-(\bdv^{T}\BfQ\bdx)^{2}/(\bdx^{T}\BfQ\bdx)\big);    
\end{align} 
otherwise, we have $\CI_{R}(\bdx;\bdv)=0$. 
\end{lemma}

\begin{proof}
If $\BfQ\bdx$ and $\BfQ\bdv$ are collinear, 
we have $\Lambda(\bdx;\bdv)=0$  
and the result follows from Lemma~\ref{lem:MinRSpanZeroDir}. 
Assume that $\BfQ\bdx$ and $\BfQ\bdv$ are non-collinear. 
From Theorem~\ref{thm:RmapOptSS},  
if ${\Upsilon(\bdv;\bdx)>0}$, then 
$\CI_{R}(\bdx;\bdv)=R(\bdx)-R(\bdz_{\tau})$, 
and a direct computation gives \eqref{eq:OSSImprovVal}. 
We conclude by observing that 
if $\bdv\in\BR^{N}$ is such that $\Upsilon(\bdv;\bdx)\leqslant 0$, 
as $\Upsilon(\bdv+\beta\bdx;\bdx)=\Upsilon(\bdv;\bdx)-\beta\Upsilon(\bdx;\bdv)$, 
there exists $\beta\in\BR$ such that $\Upsilon(\bdv+\beta\bdx;\bdx)>0$;  
as $\vspan\{\bdx,\bdv+\beta\bdx\}=\vspan\{\bdx,\bdv\}$, 
the result follows. 
\end{proof}

\begin{remark}[Exact line search for the minimisation of $D$]\label{rem:ELSForMiniD}
For $\bdx$ and $\bdv\in\BR^{N}$, $\bdv\not\in\SZ$, 
the map $t\mapsto D(\bdx+t\bdv)$, $t\in\BR$, is minimum at 
$t=\rho$, with 
\begin{equation}\label{eq:OSS4D}
\rho=-\big(\bdv^{T}(\BfQ\bdx-\bdc)\big)\big/(\bdv^{T}\BfQ\bdv). 
\end{equation}
Setting 
$\CI_{D}(\bdx;\bdv)
=D(\bdx)-\min\limits_{t\in\BR}D(\bdx+t\bdv)$, 
$\bdx$ and $\bdv\in\BR^{N}$, 
we thus have 
\begin{align*}
\CI_{D}(\bdx;\bdv)
=\bigg\{
\begin{array}{l}
\text{$\big(\bdv^{T}(\BfQ\bdx-\bdc)\big)^{2}\big/(\bdv^{T}\BfQ\bdv)$ if $\bdv\not\in\SZ$,  }\\
\text{$0$ otherwise}. 
\end{array}
\end{align*}
For $\bdx\in\coneC$ and $\bdv\in\BR^{N}$ 
such that $\BfQ\bdx$ and $\BfQ\bdv$ are non-collinear, 
we obtain 
\[
\CI_{R}(\bdx;\bdv)
=\CI_{D}(s_{\bdx}\bdx;\bdv)\correc(\bdx;\bdv), 
\text{ with }
\correc(\bdx;\bdv)=
\bigg(1-\frac{(\bdv^{T}\BfQ\bdx)^{2}}{(\bdv^{T}\BfQ\bdv)(\bdx^{T}\BfQ\bdx)} \bigg)^{-1}.   
\]
The improvement $\CI_{R}(\bdx;\bdv)$ 
is thus the product between $\CI_{D}(s_{\bdx}\bdx;\bdv)$,  
that is, the improvement yielded by an exact line search for the minimisation of $D$ from $s_{\bdx}\bdx$ along $\bdv$, 
and $\correc(\bdx;\bdv)$, the latter term accounting for the optimal rescaling of the iterate. 
Observe that $\correc(\bdx;\bdv)\geqslant 1$,
and that this term is of the form $1/\sin^{2}(\theta)$, 
with $\theta$ the angle formed by the potentials $\BfQ\bdx$ and $\BfQ\bdv$ in the RKHS $\CH$ 
(so, the more aligned the two potentials are, the larger the value of $\correc$). 
The impact of $\correc$ on the convergence of coordinate-descent-type
strategies for the minimisation of $R$ is discussed in Section~\ref{sec:CorrecAcc}. 
We may also observe that 
$D(\bdx)-R(\bdx)=\CI_{D}(\bdx;\bdx)$, 
$\bdx\in\coneC$. 
\fin
\end{remark}

\begin{remark}\label{rem:DOnElips}
Through the rescaling of $\bdx$ into $s_{\bdx}\bdx$, 
the minimisation of $R$ over $\coneC$ relates to 
the minimisation of $D$ over the set
\begin{align*}
\{\bdx\in\BR^{N}|s_{\bdx}=1\}
&=\coneC\cap\{\bdx\in\BR^{N}|~\bdx^{T}\BfQ\bdx=\bdc^{T}\bdx\}\\
&=\big\{\bdx\in\BR^{N}\big|~\|2\BfQ\bdx-\bdc\|_{\BfQ^{\dag}}^{2}=\|\bdc\|_{\BfQ^{\dag}}^{2}\big\}\backslash\SZ, 
\end{align*}
corresponding to a truncated ellipsoid in $\BR^{N}$. 
\fin
\end{remark}

\section{Coordinate descent with gradient-based rules}
\label{sec:CoordDescGradRule} 
We now compare the unconstrained minimisation of $D$ and $R$ 
using exact CD (that is, iterations consist of exact line searches along
directions in $\{\bde_{i}\}_{i\in[N]}$). 
For the coordinate selection, 
we consider gradient-based rules (other rules, such as cyclic or randomised, could be considered). 
For simplicity and without loss of generality, we assume that 
${\BfQ_{i,i}>0}$, $i\in[N]$, so that $\{\bde_{i}\}_{i\in[N]}\cap\SZ=\emptyset$. 
We set 
\[
\iota_{\BfQ}=\frac{\lambda_{\min}(\BfQ)}{N\max_{i\in[N]}\BfQ_{i,i}}\in(0,1],   
\]
with $\lambda_{\min}(\BfQ)>0$ the smallest non-zero eigenvalue of $\BfQ$ 
(see Remark~\ref{rem:OtherIotaQ}). 

\begin{remark}\label{rem:AboutCoordDesc}
From a numerical standpoint, 
a benefit from considering CD rather than 
(conjugate) gradient descent
for the minimisation of $D$ or $R$ is the affordability of the iterations. 
Setting $\bdx_{\smwhtsquare}=\bdx+t\bde_{i}$, $t\in\BR$, 
we indeed have 
\[
\bdc^{T}\bdx_{\smwhtsquare}=\bdc^{T}\bdx+ t c_{i},
\BfQ\bdx_{\smwhtsquare}=\BfQ\bdx + t\BfQ_{\bullet,i}
\text{ and }\bdx_{\smwhtsquare}^{T}\BfQ\bdx_{\smwhtsquare}=
\bdx^{T}\BfQ\bdx + t^{2}\BfQ_{i,i} + 2 t [\BfQ\bdx]_{i};   
\]
the update of the relevant quantities for implementing CD therefore 
has an $\CO(N)$ worst-case time complexity, 
against $\CO(N^{2})$ for gradient descent. 
\fin
\end{remark}

\subsection{Preliminaries: minimisation of $D$}
\label{sec:PrelimMiniD} 
Considering the unconstrained minimisation of $D$ via exact CD, 
following Remark~\ref{rem:ELSForMiniD}, 
at $\bdx\in\BR^{N}$, 
a natural rule for the selection of a coordinate is 
\begin{equation}\label{eq:BIRuleD}
i_{D,\BI,\bdx}
\in\arg\max_{i\in[N]}\CI_{D}(\bdx;\bde_{i}),  
\end{equation}
that is, we consider the coordinate leading to 
the \emph{best improvement} (BI) of $D$. As 
\[
\CI_{D}(\bdx;\bde_{i})
=\frac{[\BfQ\bdx-\bdc]_{i}^{2}}{\BfQ_{i,i}}
=\mip[\Big]{\frac{\BfQ\bde_{i}}{\|\BfQ\bde_{i}\|_{\BfQ^{\dag}}},\BfQ\bdx-\bdc}_{\BfQ^{\dag}}^{2},  
\]
the \emph{BI rule} \eqref{eq:BIRuleD} is equivalent to
the \emph{Gauss-Southwell-Lipschitz} (GSL, see \cite{nutini2015coordinate}) rule; 
notably, the \emph{BI coordinate} for $D$ at $\bdx$ also corresponds to
the \emph{coordinate potential} $\BfQ\bde_{i}$, $i\in[N]$, 
that aligns the most with $\nabla D(\bdx)=2(\BfQ\bdx-\bdc)$ in the RKHS $\CH$ 
(we shall in this case use the terminology \emph{$\CH$ coordinate}, 
see Section~\ref{sec:MiniR}). 
As a technical remark, in \eqref{eq:BIRuleD}, in case of non-unicity,  
a coordinate is simply picked at random among the arguments of the maximum 
(a similar remark holds for all the considered coordinate-selection rules). 

The following Lemma~\ref{lem:CoordImprovBoundD} 
is instrumental in proving the convergence of 
exact CD strategies with gradient-based rules 
for the minimisation of $D$ or $R$. 

\begin{lemma}\label{lem:CoordImprovBoundD}
For all $\bdx\in\BR^{N}$, we have 
$\CI_{D}\big(\bdx;\bde_{i_{D,\BI,\bdx}}\big)\geqslant \iota_{\BfQ}D(\bdx)$. 
\end{lemma}

\begin{proof}
We have
\begin{align}
\CI_{D}\big(\bdx;\bde_{i_{D,\BI,\bdx}}\big)
&=\mip[\Big]{\frac{\BfQ\bde_{i_{D,\BI,\bdx}}}{\|\BfQ\bde_{i_{D,\BI,\bdx}}\|_{\BfQ^{\dag}}},\BfQ\bdx-\bdc}_{\BfQ^{\dag}}^{2}
\geqslant
\frac{1}{N}\sum_{i\in[N]}
\mip[\Big]{\frac{\BfQ\bde_{i}}{\|\BfQ\bde_{i}\|_{\BfQ^{\dag}}},\BfQ\bdx-\bdc}_{\BfQ^{\dag}}^{2}
\label{eq:NinIotaQ}\\ 
&=\frac{1}{N}\sum_{i\in[N]}\frac{1}{\BfQ_{i,i}}[\BfQ\bdx-\bdc]_{i}^{2}  
\geqslant \frac{1}{N\max_{i\in[N]}\BfQ_{i,i}}\|\BfQ\bdx-\bdc\|_{\ell^{2}}^{2}\nonumber\\
&
\geqslant\frac{1}{N\lambda_{\max}(\BfQ^{\dag})\max_{i\in[N]}\BfQ_{i,i}}\|\BfQ\bdx-\bdc\|_{\BfQ^{\dag}}^{2},
\nonumber
\end{align}
the last inequality following from
$\|\bdh\|^{2}_{\BfQ^{\dag}}
=\bdh^{T}\BfQ^{\dag}\bdh
\leqslant\lambda_{\max}(\BfQ^{\dag})\|\bdh\|_{\ell^{2}}^{2}$,  $\bdh\in\CH$. 
We conclude by observing that 
$\lambda_{\min}(\BfQ)=1/\lambda_{\max}(\BfQ^{\dag})$. 
\end{proof}

\begin{remark}\label{rem:OtherIotaQ}
In Lemma~\ref{lem:CoordImprovBoundD}
and the upcoming developments, the constant $\iota_{\BfQ}$ 
may actually be replaced by 
\[
\tilde{\iota}_{\BfQ}=1/\big(N\lambda_{\max}(\BfW^{1/2}\BfQ^{\dag}\BfW^{1/2})\big)
\in(0,1], 
\] 
with $\BfW\in\BR^{N\times N}$ the diagonal matrix such that $\BfW_{i,i}=\BfQ_{i,i}$, $i\in\BN$. 
Indeed, observing that
\[
\|\bdh\|^{2}_{\BfQ^{\dag}}
\leqslant\lambda_{\max}(\BfW^{1/2}\BfQ^{\dag}\BfW^{1/2})(\bdh^{T}\BfW^{-1}\bdh), 
\bdh\in\CH, 
\]
the inequality    
$\CI_{D}\big(\bdx;\bde_{i_{D,\BI,\bdx}}\big)\geqslant \tilde{\iota}_{\BfQ}D(\bdx)$
follows directly from \eqref{eq:NinIotaQ}. We have $\tilde{\iota}_{\BfQ}\geqslant\iota_{\BfQ}$. 
The term $\iota_{\BfQ}$ is nevertheless more easily interpretable as it relates 
to a discrete approximation of the condition number of $\BfQ$. 
\fin
\end{remark}

To implement an exact CD with BI rule (BI-CD, for short) 
for the minimisation of $D$, we select an initial iterate $\bdx^{(0)}\in\BR^{N}$, 
and we set $\bdx^{(k+1)}=\bdx^{(k)}+\rho^{(k)}\bde_{i^{(k)}}$, $k\in\BNz$, 
with $i^{(k)}=i_{D,\BI,\bdx^{(k)}}$ and $\rho^{(k)}$ given by \eqref{eq:OSS4D}. 
Theorem~\ref{thm:convergenceBICooForD} below
shows the convergence of such strategies; 
this result is a special instance of 
some classical results from the literature
(see for instance \cite{nutini2015coordinate}), 
it is presented for completeness  
and to illustrate parallels with the minimisation of $R$. 

\begin{theorem}\label{thm:convergenceBICooForD}
Consider the minimisation of $D$ over $\BR^{N}$; 
the sequence of iterates $\{\bdx^{(k)}\}_{k\in\BNz}$ generated by  
an exact CD with BI rule verifies 
$\lim_{k\to\infty}D(\bdx^{(k)})=0$, with 
\[
D(\bdx^{(k)})\leqslant (1-\iota_{\BfQ})^{k}D(\bdx^{(0)}), 
k\in\BNz.  
\]
\end{theorem}

\begin{proof}
Lemma~\ref{lem:CoordImprovBoundD} gives
\[
D(\bdx^{(k+1)})\leqslant (1-\iota_{\BfQ})D(\bdx^{(k)})
\leqslant (1-\iota_{\BfQ})^{k+1}D(\bdx^{(0)}),
k\in\BNz, 
\]
as expected. 
As $\iota_{\BfQ}\in(0,1]$, the assertion $\lim_{k\to\infty}D(\bdx^{(k)})=0$ follows directly. 
\end{proof}

\subsection{Minimisation of $R$}
\label{sec:MiniR} 
Contrary to the minimisation of $D$, for $R$, the BI, GSL and $\CH$ rules differ. 
Below, we for simplicity only consider the BI and $\CH$ rules. 
For $\bdx\in\coneC$, we set 
\[
i_{R,\BI,\bdx}
\in\arg\max_{i\in[N]}\CI_{R}(\bdx;\bde_{i}),  
\text{ and }
i_{R,\CH,\bdx}
\in\arg\max_{i\in[N]}
\CI_{D}(s_{\bdx}\bdx;\bde_{i}).
\]
Observe that 
\[
\CI_{R}(\bdx;\bde_{i})
=\frac{[s_{\bdx}\BfQ\bdx-\bdc]_{i}^{2}}{\BfQ_{i,i}-[\BfQ\bdx]_{i}^{2}/(\bdx^{T}\BfQ\bdx)}
\text{ and }
\CI_{D}(s_{\bdx}\bdx;\bde_{i})
=\frac{[s_{\bdx}\BfQ\bdx-\bdc]_{i}^{2}}{\BfQ_{i,i}}; 
\]
from a numerical standpoint, the computation of the $\CH$ coordinate 
is therefore more affordable than the computation of the $\BI$ coordinate
(see Remark~\ref{rem:AmountWork}).  
The BI rule accounts for the \emph{acceleration term} $\correc$, 
while the $\CH$ rule does not (see Remark~\ref{rem:ELSForMiniD}). 
The $\CH$ coordinate for $R$ at $\bdx$ corresponds to
the coordinate potential $\BfQ\bde_{i}$, $i\in[N]$, 
that aligns the most with $\nabla R(\bdx)$ in the RKHS $\CH$. 

To implement an exact CD
with $\CH$ rule ($\CH$-CD, for short) for the minimisation of $R$, 
we select an initial iterate $\bdx^{(0)}\in\coneC$, 
and we set $\bdx^{(k+1)}=\bdx^{(k)}+\tau^{(k)}\bde_{i^{(k)}}$, $k\in\BNz$, 
with $i^{(k)}=i_{R,\CH,\bdx^{(k)}}$ and $\tau^{(k)}$ given by Theorem~\ref{thm:RmapOptSS}; 
we proceed accordingly for the BI rule. 
Importantly, to ensure that the non-trivial line searches occur 
in the framework of \mbox{Theorem~\ref{thm:RmapOptSS}-\ref{item:OSSDesc}}, 
following Corollary~\ref{cor:AboutRxLessThanRu}, 
we introduce the \emph{initialisation condition} 
\begin{align}\label{eq:IniCondGene}
R(\bdx^{(0)})\leqslant \min_{i\in[N]}R\big(\sign(c_{i})\bde_{i}\big);  
\end{align}
see Remark~\ref{rem:AboutInit}.  
Observe that 
$R\big(\sign(c_{i})\bde_{i}\big)=\min\{R(-\bde_{i}),R(\bde_{i})\}$. 

\begin{theorem}\label{thm:convergenceHCooForR}
Consider the minimisation of $R$ over $\BR^{N}$;  
the sequence of iterates $\{\bdx^{(k)}\}_{k\in\BNz}$ generated by  
an exact CD with $\CH$ rule verifies 
$\lim_{k\to\infty}R(\bdx^{(k)})=0$, with
\begin{equation}\label{eq:RLinearConvIotaQ}
R(\bdx^{(k)})\leqslant (1-\iota_{\BfQ})^{k}R(\bdx^{(0)}), 
k\in\BNz.  
\end{equation}
\end{theorem}

\begin{proof}
From Theorem~\ref{thm:RmapOptSS} and Corollary~\ref{cor:AboutRxLessThanRu}, 
observe that $\{\bdx^{(k)}\}_{k\in\BNz}\subset\coneC$;  
also, if there exists $k\in\BNz$ such that 
$\BfQ\bdx^{(k)}$ and $\BfQ\bde_{i^{(k)}}$ are collinear, 
then $R(\bdx^{(k)})=0$. 
For $k\in\BNz$, assume that $\BfQ\bdx^{(k)}$ and $\BfQ\bde_{i^{(k)}}$
are non-collinear;  
since a $\CH$ coordinate for $R$ at $\bdx\in\coneC$ is equivalent to 
a $\BI$ coordinate for $D$ at $s_{\bdx}\bdx$, 
from Remark~\ref{rem:ELSForMiniD}
and Lemma~\ref{lem:CoordImprovBoundD}, 
we get  
\begin{align*}
\CI_{R}(\bdx^{(k)};\bde_{i_{R,\CH,\bdx^{(k)}}})
&=\CI_{D}(s_{\bdx^{(k)}}\bdx^{(k)};\bde_{i_{R,\CH,\bdx^{(k)}}})
\correc(\bdx^{(k)};\bde_{i_{R,\CH,\bdx^{(k)}}}) \\ 
&\geqslant\CI_{D}(s_{\bdx^{(k)}}\bdx^{(k)};\bde_{i_{R,\CH,\bdx^{(k)}}})
\geqslant\iota_{\BfQ}D(s_{\bdx^{(k)}}\bdx^{(k)})
=\iota_{\BfQ}R(\bdx^{(k)}). 
\end{align*}
We hence obtain \eqref{eq:RLinearConvIotaQ} 
and $R(\bdx^{(k)})\to 0$, as expected. 
\end{proof}

\begin{corollary}\label{cor:BIvariant}
The assertions of Theorem~\ref{thm:convergenceHCooForR} also hold 
for an exact CD with $\BI$ rule for the minimisation of $R$. 
\end{corollary}

\begin{proof}
For all $\bdx\in\coneC$, 
we have 
$\CI_{R}(\bdx;\bde_{i_{R,\BI,\bdx}})
\geqslant \CI_{R}(\bdx;\bde_{i_{R,\CH,\bdx}})\geqslant\iota_{\BfQ}R(\bdx)$, 
so that inequality~\eqref{eq:RLinearConvIotaQ}
also holds for the $\BI$ rule. 
\end{proof}

\begin{remark}[Minimisation of $R$ and initialisation]\label{rem:AboutInit}
Following Remark~\ref{rem:AboutCoordDesc}, to promote sparsity 
while ensuring that condition~\eqref{eq:IniCondGene} holds, 
exact CDs for the minimisation of $R$ 
can be initialised at 
$\bdx^{(0)}=\sign(c_{i_{R,\BI,0}})\bde_{i_{R,\BI,0}}$, 
so that 
\begin{align*}
R(\bdx^{(0)})=\min_{i\in[N]}R\big(\sign(c_{i})\bde_{i}\big).
\end{align*}  
From a numerical standpoint, 
$i_{R,\BI,0}$ can be deduced from $\bdc$ and the diagonal of $\BfQ$, 
with an $\CO(N)$ time complexity.  
As a technical note, in the experiments of Section~\ref{sec:Experiments}, 
in order to compare various strategies 
with initialisation at $0$, we use the convention 
$\bdx^{(0)}=0$ and $\bdx^{(1)}=\sign(c_{i_{R,\BI,0}})\bde_{i_{R,\BI,0}}$. 
\fin
\end{remark}

\subsection{Asymptotic acceleration}
\label{sec:CorrecAcc}
Theorems~\ref{thm:convergenceBICooForD} and \ref{thm:convergenceHCooForR} 
provide the same upper bound $(1-\iota_{\BfQ})$ for the asymptotic convergence rates 
of the considered strategies for the minimisation of $D$ and $R$. 
These results therefore do not a priori suggest any potential asymptotic benefit in considering 
the minimisation of $R$ instead of $D$. 
Below, we illustrate that for the minimisation of $R$, 
an improved bound can be obtained by accounting for 
the asymptotic behaviour of the acceleration term $\correc$ 
introduced in Remark~\ref{rem:ELSForMiniD}. 

We introduce 
\begin{align}\label{eq:AccelInftyInf}
\correcT_{\infty}
=\min_{i\in[N]}\big(1-c_{i}^{2}/(\BfQ_{i,i}\|\bdc\|_{\BfQ^{\dag}}^{2})\big)^{-1}\geqslant 1, 
\end{align}
and for $\varepsilon>0$, we set $\correcT_{\infty,\varepsilon}=(1-\varepsilon)\correcT_{\infty}$. 

\begin{theorem}\label{thm:AsympCorr}
Consider the minimisation of $R$ over $\BR^{N}$; 
let $\{\bdx^{(k)}\}_{k\in\BNz}$ be the sequence of iterates generated by  
an exact CD with $\CH$ or BI rule. 
For $\varepsilon\in(0,1)$, there exists $k_{\varepsilon}\in\BNz$ 
such that for all $k\geqslant k_{\varepsilon}$, we have 
\[
R(\bdx^{(k)})\leqslant (1-\iota_{\BfQ}\correcT_{\infty,\varepsilon})^{(k-k_{\varepsilon})}R(\bdx^{(k_{\varepsilon})}).
\] 
\end{theorem}

\begin{proof}
From \eqref{eq:IniCondGene}, if there exists $i\in[N]$ such that 
$\bdc=\beta\BfQ\bde_{i}$, $\beta\in\BR$, 
then $R(\bdx^{(0)})=0$ and the result holds; we now assume that no such $i$ exists. 
From Theorem~\ref{thm:convergenceHCooForR},
as $R(\bdx^{(k)})\to 0$, we have 
$\BfQ\bdx^{(k)}/\|\BfQ\bdx^{(k)}\|_{\BfQ^{\dag}}\to\bdc/\|\bdc\|_{\BfQ^{\dag}}$.  
For all $i\in[N]$, we get 
\begin{align*}
\correc(\bdx^{(k)};\bde_{i})
\to\big(1-c_{i}^{2}/(\BfQ_{i,i}\|\bdc\|_{\BfQ^{\dag}}^{2})\big)^{-1}
\geqslant \correcT_{\infty}. 
\end{align*}  
Hence, for $\varepsilon>0$, there exists $k_{\varepsilon}\in\BNz$ 
such that for all $k\geqslant k_{\varepsilon}$, 
$\correc(\bdx^{(k)};\bde_{i_{D,\BI,\bdx^{(k)}}})
\geqslant \correcT_{\infty,\varepsilon}$; 
for $\varepsilon<1$, we in addition have $\correcT_{\infty,\varepsilon}>0$. 
From Remark~\ref{rem:ELSForMiniD}
and Lemma~\ref{lem:CoordImprovBoundD}, we obtain 
\begin{align*}
\CI_{R}(\bdx^{(k)};\bde_{i_{R,\CH,\bdx^{(k)}}})
&=\CI_{D}(s_{\bdx^{(k)}}\bdx^{(k)};\bde_{i_{R,\CH,\bdx^{(k)}}})
\correc(\bdx^{(k)};\bde_{i_{R,\CH,\bdx^{(k)}}}) \\ 
&\geqslant\CI_{D}(s_{\bdx^{(k)}}\bdx^{(k)};\bde_{i_{R,\CH,\bdx^{(k)}}})\correcT_{\infty,\varepsilon}
\geqslant\iota_{\BfQ}\correcT_{\infty,\varepsilon}R(\bdx^{(k)}),  
\end{align*}
completing the proof. 
Observe that $\iota_{\BfQ}\correcT_{\infty,\varepsilon}\in(0,1]$.  
\end{proof}

Theorem~\ref{thm:AsympCorr} entails that  
the asymptotic convergence rate for the minimisation of $R$ using exact 
CD with $\CH$ or BI rule is upper bounded by $(1-\iota_{\BfQ}\correcT_{\infty})$, 
against $(1-\iota_{\BfQ})$ for the minimisation of $D$. 
This result suggests that in situations where 
$\correcT_{\infty}$ is large, one can potentially achieve a faster asymptotic rate 
by minimising $R$ instead of $D$. 
Examples in which such an acceleration occurs are presented in 
Section~\ref{sec:Experiments}. 
Observe that $\correcT_{\infty}$ depends on $\|\bdc\|_{\BfQ^{\dag}}$, 
so that in practical applications, 
this value cannot be computed a priori. 
A schematic representation of the difference between the iterates 
of exact line searches for the minimisation of $D$ and $R$ 
is provided in Figure~\ref{fig:Illustr_Improv}. 

\begin{figure}[htbp]
\centering
\includegraphics[width=0.43\linewidth]{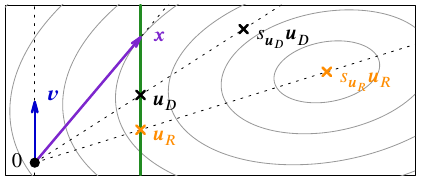}
\caption{Schematic representation of the difference between the iterates 
of exact line searches for the minimisation of $D$ and $R$. 
The line searches are implemented from $\bdx$ and along $\bdv$.  
The iterate obtained for $D$ is $\bdu_{D}$, 
and $\bdu_{R}$ is the iterate for $R$; 
the optimally rescaled iterates are also presented. 
The initial vector $\bdx$ is such that $s_{\bdx}=1$. 
The grey ellipses are level sets of the map $D$ on $\vspan\{\bdx,\bdv\}$. }
\label{fig:Illustr_Improv}
\end{figure}

\begin{remark}[Simple rescaling]\label{rem:MiniDSimpRescall}
An intermediate between the minimisation of $D$ and $R$ 
consists of minimising $D$ while systematically rescaling the iterates. 
Considering an exact CD with BI rule, 
the resulting algorithm can be described as follows: 
we select an initial vector $\bdu^{(0)}\in\BR^{N}$,  
and we set
\[ 
\text{$\bdx^{(k)}=s_{\bdu^{(k)}}\bdu^{(k)}$ 
and $\bdu^{(k+1)}=\bdx^{(k)}+\rho^{(k)}\bde_{i^{(k)}}$, $k\in\BNz$, }
\]
with $i^{(k)}=i_{D,\BI,\bdx^{(k)}}$ and $\rho^{(k)}$ given by \eqref{eq:OSS4D}. 
From Lemma~\ref{lem:CoordImprovBoundD}, we in this case have  
\begin{align*}
D(\bdx^{(k)})-D(\bdx^{(k+1)})
\geqslant D(\bdx^{(k)})-D(\bdu^{(k+1)})
\geqslant \iota_{\BfQ}D(\bdx^{(k)}), 
\end{align*}
and so $D(\bdx^{(k)})\to 0$,  
with $D(\bdx^{(k)})\leqslant (1-\iota_{\BfQ})D(\bdx^{(0)})$, $k\in\BNz$. 
As $D(\bdu^{(k+1)})\leqslant D(\bdx^{(k)})$, we also have 
$D(\bdu^{(k)})\to 0$, and so $s_{\bdu^{(k)}}\to 1$. 
Consequently, 
when compared to an exact BI-CD for the minimisation of $D$, 
the simple rescaling of the iterates does not impact 
the asymptotic convergence rate of the procedure. 
Observe nevertheless that the rescaling might affect the 
early stages of the optimisation (non-asymptotic regime);   
see Section~\ref{sec:Experiments} for illustrations.  
\fin
\end{remark}

\begin{remark}[Amount of work]\label{rem:AmountWork}
Numerically, an iteration of an exact BI-CD for the minimisation of $D$
involves finding the maximum entry of a vector of size $N$ 
(selection of the BI coordinate) and $4N+2$ flops (reducing to $3N+2$ with fused multiply-add);   
$2N$ flops are indeed needed to compute the improvement 
yielded by each coordinate (see Remark~\ref{rem:ELSForMiniD}), 
\mbox{$2$ flops} for the computation of the optimal step size $\rho$
and the update of the iterate $\bdx$, 
and $2N$ flops for the update of $\BfQ\bdx-\bdc$. 
In comparison,  
implementing an exact $\CH$-CD for the minimisation of $R$
requires $2N+13$ additional flops per iteration 
(reducing to $N+9$ flops with fused multiply-add); 
indeed, $4$ additional flops are required to compute 
$\tau$ (the optimal step size for $R$) instead of $\rho$, 
updating $s_{\bdx}$ takes $9$ flops 
($6$ flops for the update of $\bdx^{T}\BfQ\bdx$, $2$ for $\bdc^{T}\bdx$, and a division), 
and the update of  $\BfQ\bdx$ and $s_{\bdx}\BfQ\bdx-\bdc$ 
involves $2N$ additional flops (more precisely, this update involves $4N$ flops, 
against $2N$ flops for the update of $\BfQ\bdx-\bdc$ when minimising $D$).  
Implementing \mbox{BI-CD} instead of $\CH$-CD for the minimisation of $R$ 
requires $3N$ additional flops per iteration 
(computation of $i_{R,\BI,\bdx}$ instead of $i_{R,\CH,\bdx}$; reducing to $2N$ with fused multiply-add). 
\fin
\end{remark}

\section{Experiments}
\label{sec:Experiments}
We now illustrate the behaviour of the discussed strategies on a series of examples. 
For the minimisation of $D$, we consider CG (referred to as CG-$D$ in the figures), 
as well as CD with BI rule (CD-$D$) and its simple-rescaling variant 
(SR-$D$; see Remark~\ref{rem:MiniDSimpRescall}). 
For the minimisation of $R$, we consider CD with $\CH$ and BI rules 
(in the figures, we use the notations $\mathcal{H}$-$R$ and BI-$R$, respectively). 
All the minimisations are initialised at $0$; see Remark~\ref{rem:AboutInit}. 
The various strategies are compared in terms of their number of matrix-column calls, 
so that a single CG iteration is compared with $N$ CD iterations. 
When minimising $R$, we report the evolution of 
$D(s_{\bdx^{(k)}}\bdx^{(k)})=R(\bdx^{(k)})$, $k\in\BNz$.  

The focus of the forthcoming experiments is to demonstrate 
the existence of situations where the acceleration mechanisms discussed in the previous sections 
(that is, asymptotic acceleration and acceleration induced by simple rescaling) 
materialise. To further illustrate the asymptotic acceleration phenomenon 
(see Section~\ref{sec:CorrecAcc}), 
following \eqref{eq:AccelInftyInf}, we introduce 
\begin{align*}
\correcT_{\infty}^{\up}
=\max_{i\in[N]}\big(1-c_{i}^{2}/(\BfQ_{i,i}\|\bdc\|_{\BfQ^{\dag}}^{2})\big)^{-1},  
\end{align*}
corresponding to the maximum asymptotic value of the acceleration term $\correc$
induced by the minimisation of $R$ instead of $D$ via CD with gradient-based rules 
(see Remark~\ref{rem:ELSForMiniD}). 

In all this section, for a matrix $\BfX\in\BR^{N\times m}$, 
the notation $\BfX \sim U(a,b)$ indicates that the entries of $\BfX$ are independent realisations 
of a random variable with uniform distribution on the interval $[a,b]$, $a$ and $b\in\BR$; 
a similar convention holds for vectors (case $m=1$). 
We denote by $\BfI_{N}$ the $N\times N$ identity matrix, 
and we set $\one_{N}=(1)_{i\in[N]}$. 
Unless otherwise stated (see Example~\ref{ex:Examp6}), we use $N=500$. 

\begin{example}\label{ex:Examp1}
We set $\BfQ=\BfX\BfX^{T}$, with $m=250$ and  $\BfX \sim U(2,4)$;   
to ensure the existence of solutions, we also set $\bdc=\BfQ\alpb$, with $\alpb \sim U(-2,2)$. 
The values of the parameters were chosen so that the resulting random quadratic maps exhibit 
a variety of values of $\correcT_{\infty}$. 
The results are presented in Figure~\ref{fig:decaywithaccele}.   
For the considered minimisation strategies,  
the decay of $D$ as a function 
of the number of iterations (compared in terms of the number of matrix-column calls)  
is displayed for two different random quadratic maps.  
In the case $\correcT_{\infty}\approx 1.03$, we observe no noticeable 
differences between the various CD strategies considered 
for the minimisation of $D$ or $R$. 
In the case $\correcT_{\infty}\approx 18.33$, 
the CDs applied to the minimisation of $R$ converge significantly 
faster than their counterparts applied to the minimisation of $D$, 
supporting the conclusions drawn from Theorem~\ref{thm:AsympCorr} 
(that is, asymptotic acceleration); notably, 
the minimisation of $R$ via CD in this case outperforms CG 
(in the considered iteration range and when compared in terms of the number of matrix-column calls). 
Figure~\ref{fig:decaywithaccele} also displays 
the empirical distributions (kernel-density estimates) 
of $\mathfrak{a}_{\infty}$ and $\mathfrak{a}_{\infty}^{\text{up}}$ 
obtained from $10{,}000$ randomly-generated quadratic maps, 
illustrating that the asymptotic-acceleration phenomenon is relatively 
common in the considered setting. 
In Figure~\ref{fig:highaccele}, we illustrate that in the framework of Example~\ref{ex:Examp1}, 
using $\alpb \sim U(0,1)$ instead of $\alpb \sim U(-2,2)$ makes the
asymptotic-acceleration phenomenon even more prevalent. 
\end{example}  

\begin{figure}[!ht]
\centering
\includegraphics[width=0.9\linewidth]{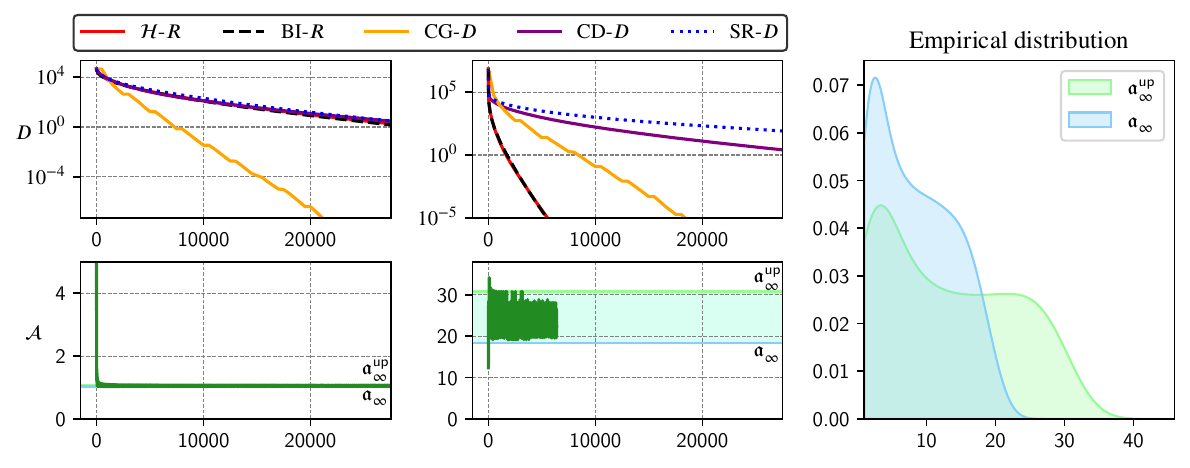}
\caption{For $\BfQ=\BfX\BfX^{T}$, $m=250$, $\BfX \sim U(2,4)$ and $\alpb \sim U(-2,2)$  
(see Example~\ref{ex:Examp1}), and for the considered minimisation strategies,   
decay of $D$ as a function of the number of iterations 
for two different random quadratic maps
(top-left and top-middle). 
For these two maps, and for the minimisation of $R$ via \mbox{$\CH$-CD}, 
the evolution of the acceleration term $\correc$ is displayed (bottom-left and bottom-middle).  
The empirical distributions of the terms 
$\mathfrak{a}_{\infty}$ and $\mathfrak{a}_{\infty}^{\text{up}}$ 
for $10{,}000$ randomly-generated quadratic maps are also presented (right). } 
 \label{fig:decaywithaccele}
\end{figure}

\begin{figure}[!ht]
\centering
\includegraphics[width=0.8\linewidth]{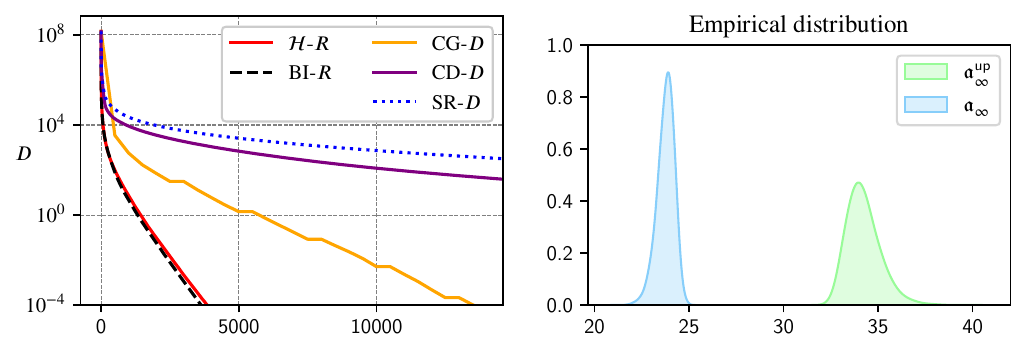}
\caption{Same setting as Figure~ \ref{fig:decaywithaccele}, but with $\alpb \sim U(0,1)$.   
For a randomly generated quadratic map, 
decay of $D$ as a function of the number of iterations for the considered minimisation strategies (left).  
The empirical distributions of $\mathfrak{a}_{\infty}$ and $\mathfrak{a}_{\infty}^{\text{up}}$ 
for $10{,}000$ random quadratic maps are also presented (right). }
\label{fig:highaccele}
\end{figure}

\begin{example}\label{ex:Examp2}
This example is a follow-up to Example~\ref{ex:Examp1}. 
We consider the same matrix $\BfX$ and vector $\bdc$ as in the top-middle plot of
Figure~\ref{fig:decaywithaccele}, but we set 
$\BfQ=\BfX\BfX^{T}+\gamma\BfI_{N}$, 
with $\gamma>0$. We use $\gamma=0.5$, $5$ and $50$;  
observe that the smallest non-zero eigenvalue of $\BfX\BfX^{T}$ is approximately $14$. 
The results are presented in Figure~\ref{fig:nugget}. 
In the three considered cases, the CDs applied to the minimisation of $R$ 
outperform their counterparts applied to the minimisation of $D$ 
(with $\mathfrak{a}_{\infty}\approx 18.25, 17.53$ and $12.65$, respectively),   
and when compared to CG, the efficiency of the CDs for the minimisation of $R$ increases with $\gamma$.  
\end{example} 

\begin{figure}[htbp]
\centering
\includegraphics[width=1\linewidth]{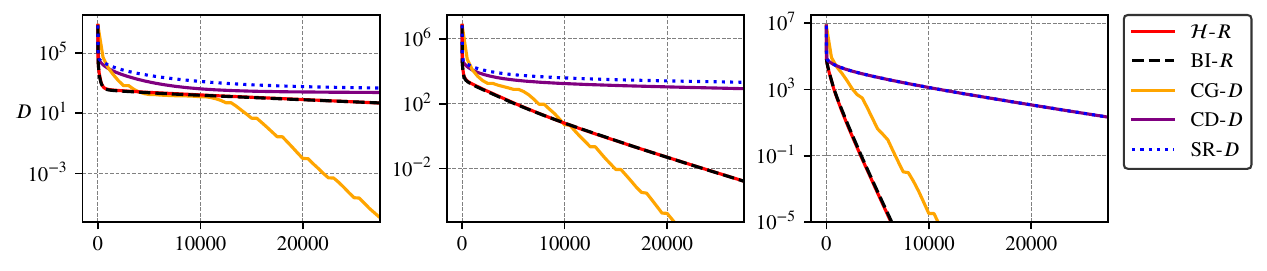}
\caption{For $\BfQ=\BfX\BfX^{T}+\gamma\BfI_N$, 
with $\gamma=0.5$, $5$ and $50$ (left to right), 
decay of the map $D$ as a function of the number of iterations 
for the considered minimisation strategies (see Example~\ref{ex:Examp2}).  
The matrix $\BfX$ and the vector $\bdc$ are the same as the ones used 
in the top-middle plot of Figure~\ref{fig:decaywithaccele}. }
\label{fig:nugget}
\end{figure}

\begin{example}\label{ex:Examp3}
We set $\BfQ=\BfX\BfX^{T}$ with $m=650$ and $\BfX \sim U(2,4)$. 
For the same realisation of $\BfQ$, we define three different solution vectors 
$\alpb$, with sparsity levels $50\%$ (that is, $50\%$ of the entries of $\alpb$ are $0$), 
$70\%$ and $90\%$, respectively; 
the non-zero entries of $\alpb$ are randomly generated from $U(-2,2)$,  
and we set $\bdc=\BfQ\alpb$. 
The values of the corresponding asymptotic-acceleration terms are
$\mathfrak{a}_{\infty}\approx 16.00$, $8.91$ and $7.81$, respectively. 
The results are presented in Figure~\ref{fig:sparse}. 
In the displayed iteration range, 
we observe that the sparser the solution is, 
the more the CDs compare positively against CG. 
For the three sparsity levels considered, 
the CDs applied to the minimisation of $R$ 
outperform their counterparts applied to the minimisation of $D$. 
\end{example}

\begin{figure}[!ht]
\centering
\includegraphics[width=1\linewidth]{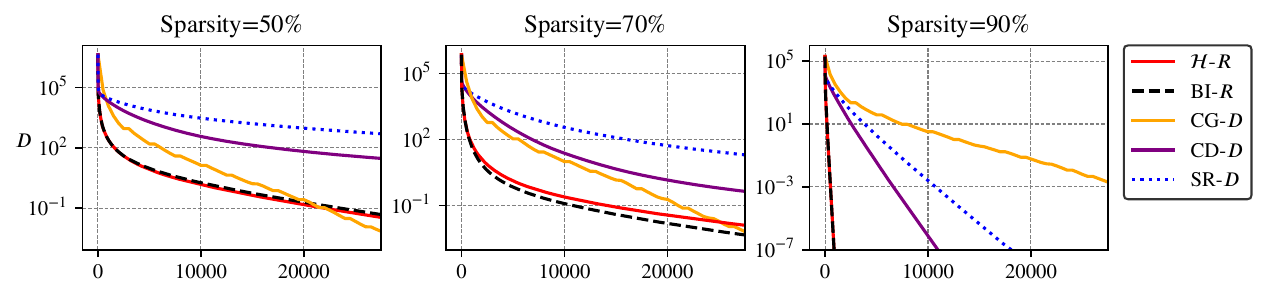}
\caption{For $\BfQ=\BfX\BfX^{T}$, with $m=650$ and $\BfX \sim U(2,4)$, 
and for varying sparsities of the solution $\alpb$ ($50\%$, $70\%$ and $90\%$; left to right), 
decay of $D$ as a function of the number of iterations 
for the considered minimisation strategies.   
The non-zero entries of $\alpb$ are generated from $U(-2,2)$; 
see Example~\ref{ex:Examp3}. }
\label{fig:sparse}
\end{figure}

\begin{example}\label{ex:Examp4}
We set $\BfQ=\BfX\BfX^{T}$ with $m=650$, $\BfX \sim U(-2,3)$ and ${\bdc \sim U(3,5)}$;   
the considered quadratic map verifies
$\correcT_{\infty} \approx \correcT_{\infty}^{\up} \approx 1$. 
The decay of $D$ for the considered minimisation strategies
is displayed in Figure~\ref{fig:rescale}. 
In this example, the CDs for the minimisation of $R$
behave similarly to a CD for the minimisation of $D$ with simple rescaling,   
so that the observed acceleration is induced by the simple-rescaling mechanism alone. 
\end{example}

\begin{figure}[!htbp]
\centering
\includegraphics[width=0.6\linewidth]{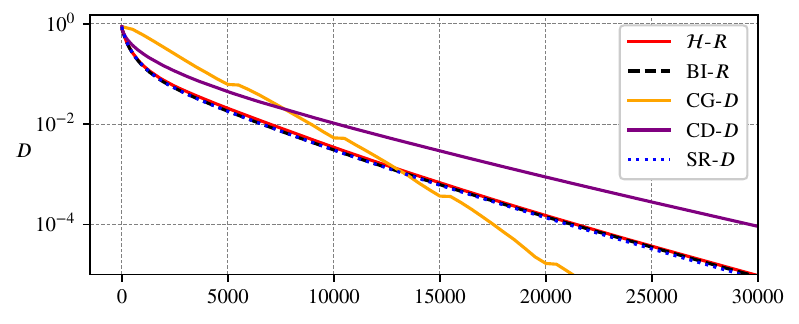}
\caption{
For $\BfQ=\BfX\BfX^{T}$, with $m=650$, $\BfX \sim U(-2,3)$ and $\bdc \sim U(3,5)$, 
decay of $D$ as a function of the number of iterations for the considered minimisation strategies; 
see Example~\ref{ex:Examp4}. }
\label{fig:rescale}
\end{figure}

\begin{example}\label{ex:Examp5}
We consider a set of $N=500$ random points uniformly distributed in $[0,1]^{5}$; 
we denote by $\BfK$ the kernel matrix defined by this set of points 
and a kernel of the form $(x,x')\mapsto\exp(-\zeta\|x-x'\|^{2})$, with $\zeta=0.13$ 
(Gaussian kernel; $\|.\|$ is the Euclidean norm of $\BR^{5}$);   
we also define $\bdv\sim U(0,1)$. We then set 
\[
\BfQ=\BfK+\gamma\BfI_{N}+\beta\one_{N}\one_{N}^{T}
\quad\text{and}\quad
\bdc=\bdv+\delta\one_{N}, 
\]
with $\gamma=1$ 
and either $\beta=\delta=0$ (Case~1), 
$\beta=0$ and $\delta=100$ (Case~2), or 
$\beta=1$ and $\delta=100$ (Case~3). 
The values of the corresponding acceleration terms are 
$\correcT_{\infty}\approx 1.00$, $1.62$ and $2.39$, respectively.  
The results are presented in Figure~\ref{fig:additive_consts}. 
We observe that when compared to the other strategies considered, 
the efficiency of the CDs for the minimisation of $R$
increases with the value of $\correcT_{\infty}$. 
\end{example}

\begin{figure}[!h]
\centering
\includegraphics[width=1\linewidth]{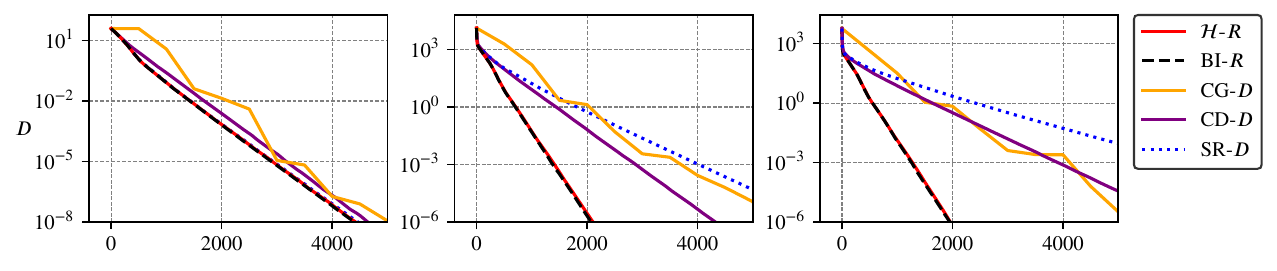}
\vspace*{-0.5cm}
\caption{Decay of the map $D$ as a function of the number of iterations for the considered minimisation strategies. 
The three underlying cases are described in Example~\ref{ex:Examp5} (Cases~1, 2 and 3, left to right). }
\label{fig:additive_consts}
\end{figure}

\begin{example}\label{ex:Examp6}
We consider the sparse matrix $\BfM$ corresponding to the 
\texttt{1138\_bus} instance of the SuiteSparse Matrix Collection; 
see \cite{davis2011university}. We have $N=1{,}138$. 
We set $\BfQ=\BfM+\gamma\BfI_{N}$ with $\gamma = 1$, and $\bdc\sim U(-1,1)$.  
The results are presented in Figure~\ref{fig:powernetworkproblem}.  
In this particular example, we have 
$\mathfrak{a}_{\infty}\approx\mathfrak{a}_{\infty}^{\text{up}}\approx 1$; 
nevertheless, when compared to a CD for the minimisation of $D$, 
the CDs for the minimisation of $R$ display an acceleration 
induced by simple rescaling. 
\end{example}

\begin{figure}[!h]
\centering
\includegraphics[width=0.6\linewidth]{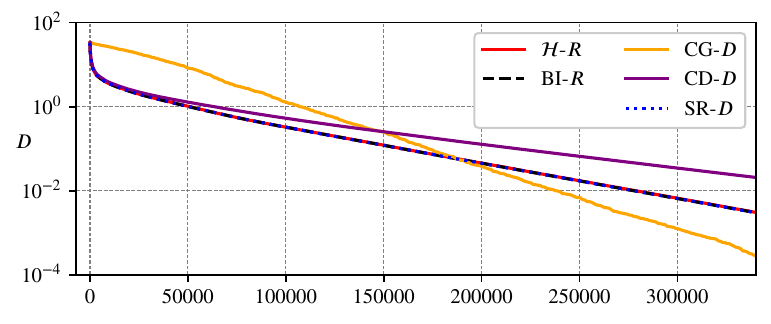}
\vspace*{-0.5cm}
\caption{Decay of the map $D$ as a function of the number of iterations for the various minimisation strategies considered. 
The underlying quadratic map is defined from a sparse symmetric positive-definite matrix; 
see Example~\ref{ex:Examp6}. }
\label{fig:powernetworkproblem}
\end{figure}

\section{Discussion}
\label{sec:ConcluDiscuss}
We investigated the properties of a class of piecewise-fractional maps arising 
from the introduction of an invariance under rescaling into convex quadratic maps,  
and studied the minimisation of such relaxed maps via CDs with gradient-based rules. 
In this setting, and when compared to equivalent CD strategies for quadratic minimisation,
we provided theoretical and empirical evidence supporting the existence, under specific conditions, 
of an asymptotic-acceleration phenomenon inherent to the minimisation of the relaxed maps   
(see Theorem~\ref{thm:AsympCorr} and Section~\ref{sec:Experiments}).    
Our experiments also demonstrate that the minimisation of the relaxed maps can in some circumstances 
benefit from a non-asymptotic acceleration analogous to the rescaling 
of the iterates of CDs for quadratic minimisation (see Remark~\ref{rem:MiniDSimpRescall}). 
From a computational standpoint and according to our investigations, 
the efficiency, in terms of number of iterations, 
of the considered CDs for the minimisation of the relaxed maps 
appears to be at worst comparable to their counterparts applied to quadratic minimisation, 
and the gains induced by minimising the relaxed maps 
sometimes substantially outweigh 
the cost of the additional operations involved (see Remark~\ref{rem:AmountWork}).   
Gaining a deeper understanding of the properties of the quadratic maps 
for which such accelerations materialise could help further support the deployment of the proposed methodology.    
The extent to which the discussed accelerations pertain in a stochastic setting 
(randomised CDs) could also be an interesting avenue for future research. 

\appendix 

\section{Proof of Theorem~\ref{thm:RmapPseudoCvx}}
\label{sec:ProofThmPseudoCvx}
\begin{proof}
For $\bdz=(1-t)\bdx+t\bdu$, 
$\bdx$ and $\bdu\in\BR^{N}$, $t\in[0,1]$, 
there always exists $s\geqslant 0$ and $t'\in[0,1]$ such that  
${s\bdz=(1-t')s_{\bdx}\bdx+t's_{\bdu}\bdu}$. Indeed,  
\begin{itemize}
\item for $\bdx\not\in\coneC$ and $\bdu\not\in\coneC$, the condition is verified for $s=0$ and any $t'\in[0,1]$; 
\item for $\bdx\not\in\coneC$ and $\bdu\in\coneC$, the condition is verified for $s=0$ and $t'=0$; 
\item for $\bdx\in\coneC$ and $\bdu\not\in\coneC$, the condition is verified for $s=0$ and $t'=1$; 
\item for $\bdx\in\coneC$ and $\bdu\in\coneC$, we have $\coni\{\bdx,\bdu\}=\coni\{s_{\bdx}\bdx,s_{\bdu}\bdu\}$. 
\end{itemize}
From the definition of $R$ and the convexity of $D$, we obtain
\[
R(\bdz)
\leqslant D(s\bdz)
\leqslant (1-t')D(s_{\bdx}\bdx) + t' D(s_{\bdu}\bdu)
=(1-t')R(\bdx) + t' R(\bdu)
\leqslant\max\{R(\bdx), R(\bdu)\}, 
\]
and $R$ is therefore quasiconvex on $\BR^{N}$. 

Let $\bdx$ and $\bdu\in\coneC$ be such that $\Lambda(\bdx;\bdu-\bdx)\geqslant 0$;  
as $\bdx^{T}(s_{\bdx}\BfQ\bdx-\bdc)=0$, 
the latter condition reads $\bdu^{T}(s_{\bdx}\BfQ\bdx-\bdc)\geqslant0$, and so 
\begin{equation}\label{eq:condiGradPseuso1}
(\bdc^{T}\bdx)(\bdu^{T}\BfQ\bdx)\geqslant (\bdx^{T}\BfQ\bdx)(\bdc^{T}\bdu).   
\end{equation}
As $\bdx$ and $\bdu\in\coneC$, we have $\bdc^{T}\bdx>0$, $\bdx^{T}\BfQ\bdx>0$ and $\bdc^{T}\bdu>0$,  
so that \eqref{eq:condiGradPseuso1} entails ${\bdu^{T}\BfQ\bdx>0}$; 
we also have $\bdu^{T}\BfQ\bdu>0$.   
By CS, we get 
$(\bdu^{T}\BfQ\bdx)^{2}\leqslant (\bdx^{T}\BfQ\bdx)(\bdu^{T}\BfQ\bdu)$, 
which combined with \eqref{eq:condiGradPseuso1} gives 
\[
\frac{(\bdc^{T}\bdx)^{2}}{(\bdx^{T}\BfQ\bdx)^{2}} 
\geqslant \frac{(\bdc^{T}\bdu)^{2}}{(\bdu^{T}\BfQ\bdx)^{2}}
\geqslant \frac{(\bdc^{T}\bdu)^{2}}{(\bdx^{T}\BfQ\bdx)(\bdu^{T}\BfQ\bdu)}. 
\]
We hence obtain $(\bdc^{T}\bdu)^{2}/(\bdu^{T}\BfQ\bdu)\leqslant (\bdc^{T}\bdx)^{2}/(\bdx^{T}\BfQ\bdx)$,
that is, $R(\bdx)\leqslant R(\bdu)$, 
and $R$ is therefore pseudoconvex on $\coneC$. 
\end{proof}

\vspace{0.5\baselineskip}
\noindent\textbf{Funding:} Alexandra Zverovich thankfully acknowledges support from the Additional Funding Programme for Mathematical Sciences, delivered by EPSRC (EP/V521917/1) and the Heilbronn Institute for Mathematical Research.

\bibliographystyle{plain}
\bibliography{Ref_Quad_Rescale}
\end{document}